\documentclass[preprint,12pt]{elsarticle}
\usepackage{amssymb,graphics,color,amsmath}
\usepackage[latin1]{inputenc}
\usepackage[pagebackref=false]{hyperref}
\newtheorem{thm}{Theorem}[section]
\newtheorem{cor}[thm]{Corollary}

\newtheorem{lem}[thm]{Lemma}
\newtheorem{prop}[thm]{Proposition}

\newtheorem{question}[thm]{Questions}

\newtheorem{defn}[thm]{Definition}
\newtheorem{example}[thm]{Example}

\newproof{proof}{Proof}

\def\N{{\mathbb{N}}}
\def\Q{{\mathbb{Q}}}
\def\Z{{\mathbb{Z}}}

\newcommand{\bu}{\mathbf{u}}
\newcommand{\bv}{\mathbf{v}}
\newcommand{\bi}{\mathbf{i}}
\newcommand{\bY}{\mathbf{Y}}
\newcommand{\bZ}{\mathbf{Z}}

\renewcommand{\em}{\textit}
\DeclareMathOperator{\Apr}{Apr}
\DeclareMathOperator{\adj}{adj}

\def\uAlg#1#2{\kappa[\![\bu^{#1}]\!]/\kappa[\![\bu^{#2}]\!]}

\def\uRing#1{\kappa[\![\bu^{#1}]\!]}
\def\vRing#1{\kappa[\![\bv^{#1}]\!]}
\def\2uRing#1#2{\kappa[\![\bu^{#1},\bu^{#2}]\!]}
\def\twovRing#1#2{\kappa[\![\bv^{#1},\bv^{#2}]\!]}
\def\threevRing#1#2#3{\kappa[\![\bv^{#1},\bv^{#2},\bv^{#3}]\!]}
\def\3uRing#1#2#3{\kappa[\![\bu^{#1},\bu^{#2},\bu^{#3}]\!]}
\def\4uRing#1#2#3#4{\kappa[\![\bu^{#1},\bu^{#2},\bu^{#3},\bu^{#4}]\!]}
\def\5uRing#1#2#3#4#5{\kappa[\![\bu^{#1},\bu^{#2},\bu^{#3},\bu^{#4},\bu^{#5}]\!]}
\def\6uRing#1#2#3#4#5#6{\kappa[\![\bu^{#1},\bu^{#2},\bu^{#3},\bu^{#4},\bu^{#5},\bu^{#6}]\!]}

\begin{document}

\bibliographystyle{elsarticle-num}

\begin{frontmatter}

\title{Factorizations in Numerical Semigroup Algebras}

\author{I-Chiau Huang }
\ead{ichuang@math.sinica.edu.tw}
\address{Institute of Mathematics, Academia Sinica,
	6F, Astronomy-Mathematics Building,
	No. 1, Sec. 4, Roosevelt Road, Taipei 10617, Taiwan, R.O.C.}

\author{Raheleh Jafari \fnref{fn1}\corref{cor2}}
\ead{rjafari@ipm.ir}
\address{Mosaheb Institute of Mathematics, Kharazmi University, and 	
	School of Mathematics, Institute for
	Research in Fundamental Sciences (IPM), P. O. Box 19395-5746,
	Tehran, Iran.}

\cortext[cor2]{Corresponding author}
\fntext[fn1]{Raheleh Jafari was in part supported by a grant from IPM (No. 95130120).}

\begin{abstract}
We study a numerical semigroup ring as an algebra over another numerical semigroup ring.
The complete intersection property of numerical semigroup algebras is investigated using
factorizations of monomials into minimal ones. The goal is to study whether a flat rectangular
algebra is a complete intersection. Along this direction, special types of algebras generated by
few monomials are worked out in detail.
\end{abstract}


\end{frontmatter}

%%% ----------------------------------------------------------------------

%%%%%%%%%%%%%%%%%%%%%%%%%%%%%%%%%%%%%%%%%%%%%%%%%%

\section{Introduction}

%%%%%%%%%%%%%%%%%%%%%%%%%%%%%%%%%%%%%%%%%%%%%%%%%%

Despite its simple definition, numerical semigroups  provide a fertile ground for research
\cite{GR}. Following \cite{HK}, we investigate algebraic properties of an exponential
counterpart of numerical semigroups from a relative point of view. Throughout this paper, 
$\kappa$ is a field. A numerical semigroup ring is a complete local domain
of the form $\kappa[\![\bu^{s_1},\ldots,\bu^{s_n}]\!]$, where $s_1, \ldots,s_n$ are 
positive \em{rational} numbers. Note that, in the literature, 
$s_1, \ldots,s_n$ are often assumed to be relatively prime positive integers. The objects we study are local 
homomorphisms $R\to R'$ of numerical semigroup rings. Through the homomorphism, 
we have an algebra $R'$ over the coefficient ring $R$, which we denote by $R'/R$. 
Within Cohen-Macaulay homomorphisms, we are mainly interested in the property of 
complete intersection. We remark that the classical study of numerical semigroup rings is 
a special case of our relative situation. Indeed, a numerical semigroup ring 
can be considered as a numerical semigroup algebra over a Noether normalization.

Following the terminology of \cite{EGA-24,EGA}, 
a numerical semigroup algebra $R'/R$ is called 
\em{Cohen-Macaulay} (resp. \em{complete intersection}) if the homomorphism $R\to R'$
is flat and its fibers are Cohen-Macaulay (resp. complete intersection) rings. There are 
only two fibers of the homomorphism. Both are zero dimensional and hence Cohen-Macaulay. 
Therefore the Cohen-Macaulay property for a numerical semigroup algebra is  simply equivalent to flatness. The notion of complete intersection has been extended to arbitrary 
homomorphisms of Noetherian rings \cite{A}. In this paper, flatness is included as a part of
the definition of complete intersection so that the hierarchy of complete intersection inside 
Cohen-Macaulay automatically holds.

In the context of numerical semigroup algebras, our definition of complete intersection takes
the form directly from the historical origin. Recall that, in algebraic geometry, a $d$-dimensional 
variety in the $n$-dimensional ambient space is complete intersection if it can be cut out 
by $n-d$ hypersurfaces. For a flat numerical semigroup algebra $R'/R$, we consider a
local surjective $R$-algebra homomorphism 
\[
\hat{\pi}\colon R[\![Y_1,\ldots,Y_n]\!]\to R'. 
\]
The power series ring $R[\![Y_1,\ldots,Y_n]\!]$ has dimension $n+1$. The kernel of $\hat{\pi}$
needs at least $n$ generators. If the kernel can be generated by $n$ elements, the flat algebra
$R'/R$ is called \em{complete intersection}. In the language of \cite{A}, 
\[
R\to R[\![Y_1,\ldots,Y_n]\!]\to R'
\]
is a Cohen factorization of $R\to R'$. If the kernel of $\hat{\pi}$ is generated by a regular 
sequence, the homomorphism $R\to R'$ is called complete intersection at the maximal ideal of 
$R'$. If $R\to R'$ is complete intersection at the maximal ideal, it is also complete 
intersection at the zero ideal \cite{T}. In \cite{A}, a homomorphism is called complete intersection if
it is complete intersection at all prime ideals. Note that the kernel of $\hat{\pi}$ is generated by a 
regular sequence if and only if it is generated by $n$ elements. Hence our definition of complete 
intersection for flat numerical semigroup algebras agrees with that of \cite{A}, and also 
with that of \cite{EGA}.

There are $n$ candidates of the form 
$Y_i^{\beta_i}-\bu^{\beta_{i0}}Y_1^{\beta_{i1}}\cdots Y_n^{\beta_{in}}$
for the set of generators of the kernel of $\hat{\pi}$, where 
$\bu^{\beta_{i0}}\in R$ and $\beta_{ii}=0$. For prescribed numerical invariants $\beta_i$, the 
image of the set $\{Y_1^{s_1}\cdots Y_n^{s_n}\,|\, 0\leq s_i<\beta_i\}$ in $R'$ can be used to 
study the complete intersection property of $R'/R$. Each monomial in the set describes a
factorization of its image. If an element of $R'$ has factorizations from two distinct monomials
in the set, the difference of the monomials creates an extra generator for the kernel of $\hat{\pi}$. 
This is the idea of $\alpha$-rectangular, $\beta$-rectangular and $\gamma$-rectangular sets 
of Ap\'{e}ry numbers in the classical case \cite{DMS-2013, DMS}.  
In the  relative situation, Ap\'{e}ry numbers are 
generalized to Ap\'{e}ry monomials; the role of these numerical invariants is replaced by the 
notion of rectangles to emphasize the ``shape'' rather than the ``size'' of the set of Ap\'{e}ry 
monomials. Investigating factorizations of a numerical semigroup algebra, 
we obtain a main result 
asserting that a flat numerical semigroup algebra is complete intersection if it has a nonsingular 
rectangle. See Theorem~\ref{thm:main}.

This paper is organized as follows. Section~\ref{sec:apr} starts with our broader
definition of numerical semigroups. Notions related to Ap\'{e}ry monomials are defined.
Section~\ref{sec:flat} reviews flatness of numerical semigroup
algebras. A new criterion for flatness is given in terms of the number of Ap\'{e}ry monomials. In
Section~\ref{sec:rectangle}, we define rectangles with examples to clarify the notion.
Sufficient conditions are given for a flat numerical semigroup algebra to be
complete intersection.
Section~\ref{sec:234} consists of a detailed study of rectangular numerical
semigroup algebras generated by few monomials. For flat algebras generated by three monomials,
we show that rectangles must be non-singular with a special type resembling free numerical semigroups. For flat algebras generated by four monomials, we find a subclass of rectangular algebras,
which consists only of complete intersection algebras.

%%%%%%%%%%%%%%%%%%%%%%%%%%%%%%%%%%%%%%%%%%%%%%%%%%

\section{Ap\'{e}ry Monomials}\label{sec:apr}

%%%%%%%%%%%%%%%%%%%%%%%%%%%%%%%%%%%%%%%%%%%%%%%%%%

In the literature, a numerical semigroup is a submonoid of the set $\N$ of non-negative integers, 
whose greatest common divisor is $1$. The condition on the greatest common divisor is equivalent 
to the statement that the complement of the semigroup in $\N$ consists of finitely 
many elements. In this classical definition, the multiplicity of the numerical semigroup 
is recognized as the smallest non-zero number in the semigroup. To study numerical semigroups,
other submonoids of $\N$ appear without the condition on greatest common divisors.
It is also observed that a numerical semigroup divided by its multiplicity naturally occurs in
the study of the tangent cones of the numerical semigroups \cite{H2015, H2016}. 
We will see that numerical semigroups in the following broader sense give rise to many flat 
numerical semigroup algebras.
\begin{defn}[numerical semigroup]
	A numerical semigroup is a monoid generated by finitely many positive rational numbers.
\end{defn}
Let $S$ be a numerical semigroup. The numerical semigroup ring $R:=\uRing{S}$ in the 
variable $\bu$ consists of power series $\sum_{s\in S}a_s\bu^s$ with $a_s\in\kappa$.   
An element $\bu^s\in R$ is called a \em{monomial} of $R$. Note that, in the notation 
$\uRing{S}$ for $R$, the numerical semigroup ring comes with a choice of variable $\bu$. 
We denote $S=\log_\bu R$. We shall allow ourself to multiply $S$ by a positive rational 
number $t$ and study the numerical semigroup ring $\vRing{tS}$. Easily tracked from the 
relation $\bu=\bv^t$, two rings $\uRing{S}$ and $\vRing{tS}$ are essentially the same.
If a numerical semigroup $S$ is generated by $s_1,\ldots,s_n$, we also write the numerical 
semigroup ring $\uRing{S}$ in terms of the chosen variable $\bu$ as 
$\kappa[\![\bu^{s_1},\ldots,\bu^{s_n}]\!]$.

Let $S$ and $S'$ be numerical semigroups. For a positive rational number $t$ satisfying
$tS\subset S'$, the numerical semigroup ring $R'=\vRing{S'}$ has an algebra structure over 
$R=\uRing{S}$ through the relation $\bu=\bv^t$. We denote the algebra by $R'/R$ and call 
it a \em{numerical semigroup algebra} with $R$ as its \em{coefficient ring}. Note that there 
are different ways to embed $R$ into $R'$. By scaling the semigroups, we may work in the 
situation where $R$ and $R'$ share the same variable $\bu$. For such a situation, we say 
$R'/R$ is a numerical semigroup algebra in the variable $\bu$. If $S'$ is generated by $S$ 
together with rational numbers $s_1,\ldots,s_n$, the numerical semigroup ring $R'$ is also 
denoted by $R[\![\bu^{S'}]\!]$ or $R[\![\bu^{s_1},\ldots,\bu^{s_n}]\!]$ to indicate its 
$R$-algebra structure. In this paper, the notation $R[\![\bu^{s_1},\ldots,\bu^{s_n}]\!]$ will
always mean a numerical semigroup algebra with the numerical semigroup ring $R$ as 
its coefficient ring.

The classical study of numerical semigroup rings fits in our relative situation by choosing 
Noether normalizations. More precisely, for a numerical semigroup ring $\uRing{S}$, we 
take an element $s\in S$ and obtain a Noether normalization $\kappa[\![\bu^s]\!]$ of 
$\uRing{S}$. The numerical semigroup ring $\uRing{S}$ is Cohen-Macaulay, Gorenstein 
or complete intersection if and only if the algebra $\uAlg{S}{s}$ has the same property \cite{HK}.

Ap\'{e}ry numbers are among the most important tools to study numerical semigroups. 
The notion can be extended to the relative situation.
\begin{defn}[Ap\'{e}ry monomial]
For a numerical semigroup algebra $R'/R$, a monomial $p$ of $R'$ is called an
Ap\'{e}ry monomial if, whenever there are monomials $p_1$ of $R$ and $p_2$ 
of $R'$ such that $p_1p_2 =p$, then $p_1=1$. 
\end{defn}
In other words, a monomial is Ap\'{e}ry if it is not divisible by any non-trivial coefficient. 
We denote the set of Ap\'{e}ry monomials of $R'/R$ by $\Apr(R'/R)$.
\begin{example}
$\Apr(\2uRing{3}{5}/\uRing{6})=\{\bu^0,\bu^3,\bu^5,\bu^8,\bu^{10},\bu^{13}\}$.
\end{example}
\begin{example}
$\Apr(\2uRing{3}{5}/\2uRing{6}{8})=\{\bu^0,\bu^3,\bu^5,\bu^{10}\}$.
\end{example}

Ap\'{e}ry numbers are the logarithmic form of Ap\'{e}ry monomials. Let $R'/R$ be a numerical semigroup 
algebra in the variable $\bu$. Let $S'=\log_\bu R'$ and $S=\log_\bu R$. Then an element $s\in S'$ such that
$\bu^s\in\Apr(R'/R)$ is called an \em{Ap\'{e}ry number} of $S'$ with respect to $S$.
In the classical case, recall that an Ap\'{e}ry number of a numerical semigroup $S'$ with respect 
to an element $m\in S'$ is defined to be an element $s\in S'$ such that $s-m\not\in S'$. In our 
terminology, these numbers are exactly Ap\'{e}ry numbers of $S'$ with respect to the numerical semigroup $m\N$.

For a numerical semigroup algebra $R'/R$, an Ap\'{e}ry monomial $p$ not equal to $1$ is called a 
\em{minimal monomial} if, whenever there are monomials $p_1$ and $p_2$ of $R'$ 
such that $p_1p_2 =p$, then one of $p_1$ and $p_2$ has to be $p$. In other words, 
minimal monomials of a numerical semigroup algebra are minimal elements among 
non-trivial Ap\'{e}ry monomials with respect to the partial order given by divisions. Let 
$\bu^{s_1},\ldots,\bu^{s_n}$ be the minimal monomials of $R'/R$, then 
$R'=R[\![\bu^{s_1},\ldots,\bu^{s_n}]\!]$. 

The following two notions are central in the
computational aspect of numerical semigroup algebras.
\begin{defn}[representation]
Let $R'/R$ be a numerical semigroup algebra in the variable $\bu$.
A representation of a monomial $\bu^s$ of $R'/R$ is an expression
$\bu^s=\bu^{s_0}\bu^w$, where $\bu^{s_0}\in R$ and $\bu^w\in\Apr(R'/R)$. 
\end{defn}
\begin{defn}[factorization]\label{fact}
Let $R'/R$ be a numerical semigroup algebra in the variable $\bu$. Let 
$\bu^{s_1},\ldots,\bu^{s_n}$ be the minimal monomials of $R'/R$. An expression 
$\bu^s=\bu^{s_0}(\bu^{s_1})^{a_1}\cdots(\bu^{s_n})^{a_n}$ for a monomial $\bu^s$ 
of $R'$, where $\bu^{s_0}\in R$ and $a_i\in\N$, is  called a \em{factorization} of 
$\bu^s$.
\end{defn}
Note that representations and factorizations always exist for monomials in a 
numerical semigroup algebra.

%%%%%%%%%%%%%%%%%%%%%%%%%%%%%%%%%%%%%%%%%%%%%%%%%%%%

\section{Flatness}\label{sec:flat}

%%%%%%%%%%%%%%%%%%%%%%%%%%%%%%%%%%%%%%%%%%%%%%%%%%%%

Flatness is a homological notion for modules. It also has a characterization in terms of 
relations. Roughly speaking, a module is flat if all its module relations come from the relations 
of the underlying ring. See \cite[Theorem 7.6]{M} for the precise statement. For a finitely generated
module over a Noetherian local ring, free and flat are equivalent properties. To emphasize the 
computational aspect of numerical semigroup algebras, we interpret this fact again as in \cite{HK}
using the flatness criterion by relations.
\begin{lem}\label{lem:ur}
	A numerical semigroup algebra is flat if and only if every monomial has a unique representation.
\end{lem}
\begin{proof}
	If every monomial has a unique representation, the numerical semigroup algebra is free
	and thus flat. Assume that the algebra $\uAlg{S'}{S}$ is flat. Consider two different representations
	$\bu^{t_1}\bu^{w_1}=\bu^{t_2}\bu^{w_2}$, where $\bu^{t_1},\bu^{t_2}$ are coefficients and
	$\bu^{w_1},\bu^{w_2}$ are Ap\'{e}ry. Assume that $t_1<t_2$. Applying \cite[Theorem 7.6]{M}
	for the flat algebra, there are elements $f_{ij}\in\uRing{S}$ and $g_j\in\uRing{S'}$ such that
	\begin{eqnarray}
	\bu^{t_1}f_{1j}-\bu^{t_2}f_{2j}&=&0,\label{eq:510891}\\
	f_{i1}g_1+\cdots+f_{in}g_n&=&\bu^{w_i},\label{eq:510772}
	\end{eqnarray}
	where $i=1,2$ and $1\leq j\leq n$ for some $n$. From (\ref{eq:510891}), we see that the
	constant term of $f_{1j}$ vanishes for each $j$. Since $\bu^{w_1}$ is an Ap\'{e}ry monomial,
	this implies that the coefficient of
	$f_{11}g_1+\cdots+f_{1n}g_n$ at $\bu^{w_1}$ vanishes, contradicting  (\ref{eq:510772}).
	\qed
\end{proof}
In particular, Ap\'{e}ry monomials form an $R$-module basis of a flat numerical semigroup algebra over $R$.
\begin{example}\label{ex:nonflat}
The algebra $\uRing{}/\2uRing{2}{3}$ 
has two Ap\'{e}ry monomials $\bu^0$ and $\bu$. The algebra
is not flat, since $\bu^{3}$ has different representations $\bu^2\bu$ and $\bu^3\bu^0$. 
\end{example}
\begin{example}\label{ex:flatclassical}
The algebra $\2uRing{2}{3}/\uRing{2}$ has two Ap\'{e}ry monomials $\bu^0$ and $\bu^3$. 
All monomials with even exponents belong to the coefficient ring $\uRing{2}$, and all monomials in $\2uRing{2}{3}$ with odd exponents are uniquely represented as $\bu^{2t}\bu^3$ for some integer $t\geq 0$. By Lemma~\ref{lem:ur}, the algebra is flat.
\end{example}
\begin{example}\label{ex:flatglue}
The algebra $\4uRing{12}{14}{16}{35}/\2uRing{12}{16}$ has four Ap\'{e}ry monomials $\bu^0,\bu^{14},\bu^{35},\bu^{49}$. Since the difference of any two distinct elements in $\{0,14,35,49\}$ is not the difference of any two elements in the numerical semigroup $\langle 12,16\rangle$, any monomial of the algebra has a unique representation. The algebra is flat by Lemma~\ref{lem:ur}.
\end{example}

Example~\ref{ex:flatclassical} is part of a general phenomenon in the classical case.
If the coefficients of a numerical semigroup algebra provide a Noether normalization,
then the algebra is flat \cite[Corollary 2.2]{HK}. 
Example~\ref{ex:flatglue} is a general phenomenon about gluing. Let $S$ and $T$ be numerical semigroups generated by integers, $q\in S$ and $p\in T$ be relatively prime numbers. 
Recall that the numerical semigroup $pS+qT$ is called a \em{gluing} of $S$ and $T$, if $p$ and $q$ are not in the minimal set of generators of $T$ and $S$, respectively. The algebra $\kappa[\![\bu^{pS+qT}]\!]$
over $\kappa[\![\bu^{pS}]\!]$ or over $\kappa[\![\bu^{qT}]\!]$ is flat \cite[Proposition 2.9]{HK}. 
See Example~\ref{ex:22} for another flat numerical semigroup algebra beyond these
general phenomena.

In this section, we provide two more criteria  for flatness:
A criterion counts Ap\'{e}ry monomials and another is expressed in terms of the set
\[
\Delta_S(S'):=\{a_1-a_2 \,|\, \text{$\bu^{a_1},\bu^{a_2}\in\Apr(\uAlg{S'}{S})$ and  $a_1\geq a_2$}\}
\]
for numerical semigroups $S\subset S'$.

\begin{prop}\label{flat}
	Let $\uAlg{S'}{S}$ be a numerical semigroup algebra, where $S'$ and $S$ are subsets of $\N$
	with greatest common divisors $d'$ and $d$ respectively. There are at least $d/d'$
	Ap\'{e}ry monomials. The following conditions are equivalent.
	\begin{itemize}
		\item The algebra is flat.
		\item The algebra has $d/d'$ Ap\'{e}ry monomials.
		\item $\Delta_S(S')\cap (d/d')\Z\subset S$.
	\end{itemize}
	In particular, if the algebra is flat, then $S=S'\cap (d/d')\Z$.
\end{prop}
\begin{proof}
	Dividing every number in $S'$ by $d'$, we may assume that $d'=1$. With this assumption,
	there exists an Ap\'{e}ry number in each congruence class modulo $d$. So totally there 
	are at least $d$ Ap\'{e}ry monomials. If there are exactly $d$ Ap\'{e}ry monomials, 
each congruence class consists of exactly one Ap\'{e}ry monomial. 
	Consider two representations $\bu^{s_1}\bu^{w_1}=\bu^{s_2}\bu^{w_2}$, where 
	$\bu^{s_1},\bu^{s_2}$ are coefficients and $\bu^{w_1},\bu^{w_2}$ are Ap\'{e}ry. Since 
	$s_1$ and $s_2$ are both divisible by $d$, the exponents $w_1$ and $w_2$ are in the
	same congruence class. Hence $\bu^{w_1}=\bu^{w_2}$. In other words, every monomial 
	has a unique representation. By Lemma~\ref{lem:ur}, the algebra is flat. If there are more 
	than $d$ Ap\'{e}ry monomials, two different Ap\'{e}ry numbers $w_1$ and $w_2$ are in 
	the same congruence class. Say $w_1>w_2$. Choose $s_1$ and $s_2$ in $S$ large 
	enough so  that $w_1-w_2=s_2-s_1$. Then we have two different representations 
	$\bu^{s_1}\bu^{w_1}=\bu^{s_2}\bu^{w_2}$ for a monomial. Consequently, the algebra is not flat.
	
	The inclusion $\Delta_S(S')\cap d\Z\subset S$ is just another way to state that each
	congruence class contains exactly one Ap\'{e}ry number. The condition that there is only
	one Ap\'{e}ry number congruent to $0$ modulo $d$ can be stated as $S'\cap d\Z=S$,
	which is therefore a necessary condition for flatness. \qed
\end{proof}

\begin{example}\label{ex:22}
	For $\bu=\bv^6$, the algebra $\twovRing{4}{9}/\2uRing{2}{3}$ is flat, since it has six Ap\'{e}ry monomials
	$1,\bv^4,\bv^8,\bv^9,\bv^{13},\bv^{17}$. Note that the numerical semigroup $\langle 4,9\rangle$ can not be written as a gluing of $\langle 2,3\rangle$ and another numerical semigroup, because the semigroup obtained from such a gluing needs at least three generators \cite[Lemma 9.8]{GR}.
\end{example}

For a numerical semigroup algebra
$R[\![\bu^{s_1},\ldots,\bu^{s_n}]\!]$, we may choose $m\in\N$ such that $ms_i\in\log_\bu R$ for all $i$.
Therefore $R[\![\bu^{s_1},\ldots,\bu^{s_n}]\!]$ can be considered as the algebra obtained from $R$
by joining $m$-th roots of the monomials $\bu^{ms_1},\ldots,\bu^{ms_n}$ in $R$.
We may apply Proposition~\ref{flat} to the case that only one root is joined.

\begin{cor}\label{cor:f1}
Assume that $\log_\bu R\subset\N$ and has greatest common divisor $1$. Assume that $s,m\in\N$
are relatively prime. Then the algebra $R[\![\bu^{s/m}]\!]$ is flat if and only if $s\in\log_\bu R$.
\end{cor}
\begin{proof}
The set of Ap\'{e}ry numbers is $\{0,s/m,2s/m,\ldots,(t-1)s/m\}$, where $t$ is the smallest positive
integer satisfying $ts/m\in\log_\bu R$. The algebra is flat if and only if there are $m$ Ap\'{e}ry numbers.
If the algebra is flat, then $t=m$ and hence $s\in\log_\bu R$. Conversely, if $s\in\log_\bu R$, then
$t\leq m$. There are at least $m$ Ap\'{e}ry monomials. Hence $t=m$ and consequently the algebra is
flat. \qed
\end{proof}
Let $S\subset\N$ be a numerical semigroup with the greatest common divisor $1$.
As an application of Corollary~\ref{cor:f1}, the algebra $\2uRing{S}{s}/\uRing{S}$
is not flat for any $s\in\N\setminus S$. See Example~\ref{ex:nonflat}. More generally, for any numerical 
semigroup $S'$ satisfying $S\subsetneq S'\subset\N$, Proposition~\ref{flat} implies that 
the algebra $\uAlg{S'}{S}$ is not flat. If we remove the condition on the
greatest common divisor, plenty of flat algebras come out by adding monomials.
\begin{example}
Let $R=\2uRing{2}{3}$. Since $3\in\log_\bu R$, the algebra $R[\![\bu^{3/2}]\!]$ is flat by 
Corollary \ref{cor:f1}. In terms of the new variable $\bv=\bu^{1/2}$, we may write 
$R[\![\bu^{3/2}]\!]=\threevRing{4}{6}{3}=\twovRing{4}{3}$.
\end{example}
\begin{example}
Joining the $4$-th root of $\bu^6$ and the square roots of $\bu^5$ and $\bu^7$ to 
$\3uRing{5}{6}{7}$, we obtain $\threevRing{3}{5}{7}$ in terms of the variable $\bv=\bu^{1/2}$.
The algebra $\threevRing{3}{5}{7}/\3uRing{5}{6}{7}$ has more than two Ap\'{e}ry monomials,
including three minimal monomials. By Proposition \ref{flat}, the algebra is not flat.
\end{example}

If $S'$ is generated by $S$ and one more element in Proposition~\ref{flat}, the condition 
$S'\cap d\Z=S$ is also sufficient for flatness as stated in Proposition~2.5 in \cite{HK}.
This is not true in general.
\begin{example}
	In the algebra $\3uRing{5}{8}{9}/\3uRing{9}{15}{21}$, 
	the monomial $\bu^{23}$ has different representations $\bu^{18}\bu^5$ and $\bu^{15}\bu^8$. 
	The algebra is not flat, although $\langle5,8,9\rangle\cap 3\Z=\langle 9,15,21\rangle$.
\end{example}

Now, we give a necessary condition for a numerical semigroup algebra
to be flat. For an element $s$ in a numerical semigroup $T$, a \em{divisor} of $s$ in $T$ is an
element $t\in T$ such that $s-t\in T$. The terminology is justified by its exponential counterpart,
where $\bu^t$ divides $\bu^s$ in $\uRing{T}$. We call $0$ the
\em{trivial divisor} of $s$ and any divisor not equal to $s$ a \em{proper divisor}.
\begin{prop}\label{prop:log}
	If a numerical semigroup algebra $R[\![\bu^{s_1},\ldots,\bu^{s_n}]\!]$  is flat, then any two minimal
	generators of $\log_\bu R$ belonging to $T:=\langle s_1,\ldots,s_n\rangle$ have only
	trivial common divisor  in $T$.	
\end{prop}
\begin{proof}
	Let $a_1,\ldots,a_m$ be the minimal generators of $\log_\bu R$. Assume that
	$a_1\in T$. We claim that any proper divisor $a'_1$ of $a_1$ in $T$ is an Ap\'{e}ry number.
	Write
	\begin{equation}\label{eq:3144}
	a'_1=\sum\beta_ia_i+\sum \gamma_js_j
	\end{equation}
	and
	\[
	a_1-a'_1+\sum\gamma_js_j=\sum\beta'_ia_i+\sum\gamma'_js_j
	\]
	for certain Ap\'{e}ry numbers $\sum \gamma_js_j$ and $\sum\gamma'_js_j$. Then
	\[
	a_1=\sum(\beta_i+\beta'_i)a_i+\sum\gamma'_js_j.
	\]
	By the criterion of flatness in Lemma~\ref{lem:ur},
	\[
	\sum(\beta_i+\beta'_i)a_i=a_1
	\]
	Since $a_1,\ldots,a_m$ are minimal generators, we have $\beta_1+\beta'_1=1$ and
	$\beta_i=\beta'_i=0$ for $i>1$. Since $a_1>a'_1$, we have $\beta_1=0$ from (\ref{eq:3144}).
	Hence $a'_1$ is the Ap\'{e}ry number $\sum \gamma_js_j$.

	For $\alpha\in\Z$, we use the notation $\alpha^+:=\alpha$ if $\alpha>0$ and $\alpha^+:=0$
	otherwise. To show the proposition, we assume the contrary that $a_1,a_2\in T$ come with
	a common non-trivial divisor $t$. By joining $t$ to the set $\{s_i\}$, we may assume $t=s_1$.
	Then there are expressions
	$a_1=\sum\alpha_is_i$ and $a_2=\sum\alpha'_is_i$ for
	positive $\alpha_1$ and $\alpha'_1$. By flatness, $\sum(\alpha'_i-\alpha_i)^+s_i$ and
	$\sum(\alpha_i-\alpha'_i)^+s_i$ in the expressions
	\[
	a_1+\sum(\alpha'_i-\alpha_i)^+s_i=a_2+\sum(\alpha_i-\alpha'_i)^+s_i
	\]
	cannot be both Ap\'{e}ry. Assume that the divisor $\sum(\alpha'_i-\alpha_i)^+s_i$ of $a_2$
	is not Ap\'{e}ry. By the claim in the previous paragraph, the number $(\alpha'_1-\alpha_1)^+$
	can not be less than $\alpha'_1$. This is impossible since $\alpha_1>0$. \qed
\end{proof}
For $n=2$ in Proposition~\ref{prop:log}, we have more precise information.
\begin{cor}\label{2.5}
	Let $s_1$ and $s_2$ be positive integers satisfying $\gcd(s_1,s_2)=1$.
The following statements are equivalent for a numerical semigroup algebra
$R[\![\bu^{s_1},\bu^{s_2}]\!]$ satisfying the condition $\log_{\bu}R\subseteq\langle s_1,s_2\rangle$.
\begin{itemize}
	\item The algebra $R[\![\bu^{s_1},\bu^{s_2}]\!]$ 	is flat.
	\item $\log_{\bu}R$ is either principal or is generated by $a_1s_1$ and $a_2s_2$ for some 	
	positive integers $a_1$ and $a_2$ such that $a_1$ divides $s_2$ and $a_2$ divides $s_1$.
\end{itemize}
\end{cor}
\begin{proof}
	We work on the case that $\log_{\bu}R$ is not principal. By Proposition~\ref{prop:log}, it is generated
	by two elements $a_1s_1$ and $a_2s_2$ with only trivial divisor in common.
	For $0\leq r_1<a_1$ and $0\leq r_2<a_2$, we claim that $r_1s_1+r_2s_2$ is Ap\'{e}ry.
	Assume the contrary, by changing indices, we may write $r_1s_1+r_2s_2=a_1s_1+r'_1s_1+r'_2s_2$.
	Then $a_2s_2=(a_1-r_1+r'_1)s_1+(a_2-r_2+r'_2)s_2$ implies that $s_1$ is a divisor of $a_2s_2$.
	This is impossible. For distinct pairs $(r_1,r_2)$ and $(r'_1,r'_2)$ satisfying $0\leq r_i<a_i$ and
	$0\leq r'_i<a_i$, we claim that $r_1s_1+r_2s_2\neq r'_1s_1+r'_2s_2$.
	We may assume $r_1<r'_1$. If $r_1s_1+r_2s_2=r'_1s_1+r'_2s_2$, then
	$a_2s_2=(r_1'-r_1)s_1+(a_2-r_2+r'_2)s_2$ implies that $s_1$ is a divisor of $a_2s_2$.
	This is impossible. Therefore there are $a_1a_2$ Ap\'{e}ry numbers.  By Proposition~\ref{flat},
	the $R$-algebra $R[\![\bu^{s_1},\bu^{s_2}]\!]$ is flat if and only if
	$\gcd(a_1s_1,a_2s_2)=a_1a_2$, equivalently $a_1$ divides $s_2$ and $a_2$ divides $s_1$. \qed
\end{proof}

\begin{example}
The algebra $\2uRing{3}{4}/\2uRing{9}{12}$ is not flat by Corollary~\ref{2.5}. The algebra has
Ap\'{e}ry monomials $1, \bu^3, \bu^4, \bu^6,\bu^7,\bu^8,\bu^{10},\bu^{11},\bu^{14}$. While $\gcd(9,12)=3$,
one sees that the algebra is not flat also by counting Ap\'{e}ry monomials.
\end{example}

%%%%%%%%%%%%%%%%%%%%%%%%%%%%%%%%%%%%%%%%%%%%%%%%%%%%

\section{Rectangles}\label{sec:rectangle}

%%%%%%%%%%%%%%%%%%%%%%%%%%%%%%%%%%%%%%%%%%%%%%%%%%%%

In this section, we denote by $\bu^{s_1},\ldots,\bu^{s_n}$ the
minimal monomials of a numerical semigroup algebra $R'/R$ in the variable $\bu$.
\begin{defn}[rectangle]
The set of the Ap\'{e}ry monomials of the algebra $R'/R$ is called a rectangle of size
$\beta_1\times\cdots\times\beta_n$ if the following conditions hold.
\begin{itemize}
\item
All Ap\'{e}ry monomials can be factored uniquely as
$(\bu^{s_1})^{\ell_1}\cdots (\bu^{s_n})^{\ell_n}$, where $0\leq \ell_i<\beta_i$.
\item
All monomials $(\bu^{s_1})^{\ell_1}\cdots (\bu^{s_n})^{\ell_n}$ with
$0\leq \ell_i<\beta_i$ are Ap\'{e}ry.	
\end{itemize}
A numerical semigroup algebra is \em{rectangular} if the set of its Ap\'{e}ry monomials forms a rectangle.
\end{defn}
As seen in the next example, the condition of uniqueness is essential in the definition of rectangles.
\begin{example}
	The algebra $\4uRing{14}{21}{22}{33}/\uRing{22}$ is flat with $22$ Ap\'{e}ry monomials. 
	Since $22\neq\beta_1\times\beta_2\times\beta_3$ for integers $\beta_i>1$, the algebra is not 
	rectangular. Note that the set of Ap\'{e}ry monomials can be described as
	$\{1,\bu^{14},\bu^{28},\bu^{42}\}\times\{1,\bu^{21},\bu^{42}\}\times\{1,\bu^{33}\}$. However this set is not 
	a rectangle, because $\bu^{42}$ and $\bu^{75}$ have different expressions.
\end{example}
A rectangle of Ap\'{e}ry monomials may have different sizes.
\begin{example}
The algebra $\2uRing{2}{3}/\uRing{12}$ is rectangular. Its set of Ap\'{e}ry monomials can be described as
a $4\times 3$ rectangle $\{1,\bu^3,\bu^6,\bu^9\}\times\{1,\bu^2,\bu^4\}$ or a $2\times 6$ rectangle
$\{1,\bu^3\}\times\{1,\bu^2,\bu^4,\bu^6,\bu^8,\bu^{10}\}$.
\end{example}
A numerical semigroup algebra obtained by joining one root of a monomial is always rectangular,
but it may not  be flat, see Example~\ref{ex:nonflat}. 
We are mainly interested in flat rectangular algebras.
Since a rectangular algebra has a unique maximal Ap\'{e}ry monomial, flat
rectangular algebras are Gorenstein  by \cite[Proposition 3.1]{HK}.
Here is another non-flat rectangular algebra:
\begin{example}\label{ex:35}
The algebra $\4uRing{14}{21}{22}{33}/\2uRing{14}{22}$ is not flat, since
$\bu^{231}$ has two representations $(\bu^{14})^{15}\bu^{21}=(\bu^{22})^9\bu^{33}$.
The algebra is rectangular and the set of its Ap\'ery monomials is $\{1,\bu^{21}\}\times\{1,\bu^{33}\}$.
\end{example}
Certain flat numerical semigroup algebras are always rectangular. See the proof of Corollary~\ref{2.5}.
\begin{prop}
Let $R$ be a numerical semigroup ring which is not a power series ring.
	A flat algebra $R[\![\bu^{s_1},\bu^{s_2}]\!]$ satisfying
	$\log_{\bu}R\subseteq\langle s_1,s_2\rangle\subseteq\N$ and $\gcd(s_1,s_2)=1$ is rectangular.
\end{prop}

Free numerical semigroups~\cite{BC} (not to be confused with the notion of free algebras 
given by numerical semigroups) provide examples of rectangular algebras in the classical 
case: Assume that $R=\kappa[\![\bu^{s_0}]\!]$. Let $S$ be the semigroup minimally 
generated by $s_0,s_1,\ldots,s_n$. Assume that $S$ is free in the sense that its Ap\'{e}ry 
numbers with respect to $s_0\N$ can be listed as $\sum_{i=1}^n\lambda_is_i$ for 
$0\leq\lambda_i<\phi_i$ after rearranging the indices, where 
$\phi_i=\min\{h\in\N\,|\, hs_i\in\langle s_0,\ldots,s_{i-1}\rangle\}$. By 
\cite[Proposition 9.15]{GR}, there are $\phi_1\phi_2\cdots\phi_n$ Ap\'{e}ry monomials.
So the algebra $\kappa[\![\bu^{s_0},\ldots,\bu^{s_n}]\!]/\kappa[\![\bu^{s_0}]\!]$ is rectangular.

 The notions of $\alpha$-rectangular, $\beta$-rectangular and 
$\gamma$-rectangular Ap\'ery sets for numerical semigroups also provide rectangular algebras.
Indeed, there are strict implications
\[
\text{$\alpha$-rectangular $\implies$ $\beta$-rectangular $\implies$ $\gamma$-rectangular
$\implies$ free}
\]
for numerical semigroups \cite[Theorem~2.13]{DMS}. In particular, if $S$ has an 
$\alpha$-rectangular Ap\'ery set, then the algebra $\uAlg{S}{s}$ is rectangular for some $s\in S$.
However, the converse is not true. For instance, the semigroup $\langle 5,6,9\rangle$ does not have 
$\gamma$-rectangular Ap\'ery set (hence not $\beta$-rectangular nor $\alpha$-rectangular), but 
$\kappa[\![\bu^5,\bu^6,\bu^9]\!]/\kappa[\![\bu^6]\!]$ is a rectangular algebra with the set  
$\{1,\bu^9\}\times\{1,\bu^5,\bu^{10}\}$ of Ap\'ery monomials.

Given a numerical semigroup $S$, the properties of Cohen-Macaulay, Gorenstein and
complete intersection of the algebra $\uAlg{S}{s}$ are independent of the choice of an
element $s\in S$. This is not the case for rectangular algebras.
\begin{example}
The Ap\'{e}ry monomials $1,\bu^2,\bu^3,\bu^4,\bu^6$ of
	$\2uRing{2}{3}/\uRing{5}$ do not form a rectangle. An algebra with only one
	minimal monomial, for instance $\2uRing{2}{3}/\uRing{3}$, is always rectangular.
\end{example}

If the set of Ap\'{e}ry monomials of $R'/R$ form a rectangle of size 
$\beta_1\times\cdots\times\beta_n$, every monomial can be
factored as $\bu^t(\bu^{s_1})^{\ell_1}\cdots (\bu^{s_n})^{\ell_n}$, where $\bu^t\in R$ and
$0\leq\ell_i<\beta_i$ for all $i$. Such a factorization is unique if and only if the algebra is flat.
Under the flatness assumption, there is a unique factorization 
\[
\bu^{s_i\beta_i}=\bu^{t_i}(\bu^{s_1})^{\beta_{i 1}}\cdots (\bu^{s_n})^{\beta_{i n}}
\]
such that $0\leq\beta_{ij}<\beta_i$ for all $j$. Note that $\beta_{ii}=0$ 
in the above factorization. Let $\bY$ be the variables 
$Y_1,\ldots,Y_n$ and $\bZ$ be the shorthand of $Z_1,\ldots,Z_n$, where 
$Z_\ell:=Y_\ell^{\beta_\ell}Y_1^{-\beta_{\ell 1}}\cdots Y_n^{-\beta_{\ell n}}$. In terms of the matrix
\begin{equation}\label{def:log}
\log_\bY\bZ:=
\begin{pmatrix}
\beta_1 & -\beta_{12} & \cdots & -\beta_{1n} \\
-\beta_{21}& \beta_2 & \cdots & -\beta_{2n} \\
\vdots & \vdots & \ddots & \vdots \\
-\beta_{n1} & -\beta_{n2} & \cdots & \beta_n
\end{pmatrix},
\end{equation}
we have a relation
\begin{equation}\label{eq:log}
\log_\bY\bZ
\begin{pmatrix} s_1 \\ s_2 \\ \vdots \\s_n \end{pmatrix}=
\begin{pmatrix} t_1 \\ t_2 \\ \vdots \\t_n \end{pmatrix}.
\end{equation}

\begin{example}
	In the algebra $\6uRing{32}{35}{38}{44}{48}{56}/\2uRing{32}{48}$, the squares of minimal
	monomials are not Ap\'{e}ry. So all Ap\'{e}ry monomials are in the set
	$\{0,\bu^{35}\}\times\{0,\bu^{38}\}\times\{0,\bu^{44}\}\times\{0,\bu^{56}\}$. Since $\gcd(32,48)=16$,
	this set consists of all the 16 Ap\'{e}ry monomials. Therefore the algebra is flat with a rectangle
	of size $2\times 2\times 2\times 2$. Relation (\ref{eq:log}) becomes
	\[
	\begin{pmatrix}
	2 & -1 & 0 & 0 \\
	0& 2 & -1 & 0 \\
	0 & 0 & 2 & -1 \\
	0 & 0 & 0 & 2
	\end{pmatrix}
	\begin{pmatrix}35\\38\\44\\56\end{pmatrix}=
	\begin{pmatrix}32\\32\\32\\2\times 32+48\end{pmatrix}.
	\]
\end{example}

If $\log_\bY\bZ$ is invertible, we call the rectangle \em{non-singular}. In such a case,
every monomial in $\bY$ can be written uniquely as $Z_1^{i_1}\cdots Z_n^{i_n}$ with
rational exponents $i_1,\ldots,i_n$. In other words, row vectors of $\log_\bY\bZ$ form
a basis for the vector space $\Q^n$. If furthermore all entries of the inverse of $\log_\bY\bZ$
are non-negative, every monomial in $\bY$ can be written as $Z_1^{i_1}\cdots Z_n^{i_n}$ with 
non-negative rational exponents. This statement can be also expressed in terms of 
vectors of exponents: If a vector in $\Q^n$ sits in the first orthant with respect to the 
standard basis, then the vector also sits in the first orthant with respect to the basis 
given by the row vectors of $\log_\bY\bZ$.

\begin{lem}\label{lem:matrix}
	Let
	\[
	B=\begin{pmatrix}
	\beta_1 & -\beta_{12} & \cdots & -\beta_{1n} \\
	-\beta_{21}& \beta_2 & \cdots & -\beta_{2n} \\
	\vdots & \vdots & \ddots & \vdots \\
	-\beta_{n1} & -\beta_{n2} & \cdots & \beta_n
	\end{pmatrix}
	\]
	be a matrix of real numbers satisfying the property that $\beta_j>\beta_{ij}\geq 0$ 
	for all $i$ and $j$. If $\beta_is_i\geq\sum_{j\neq i}\beta_{ij}s_j$ for certain positive 
	numbers $s_1,\ldots,s_n$ and for all $i$, then $\det B$ and all the entries of the adjoint of 
	$B$ are non-negative.
\end{lem}

\begin{proof}
We use induction on the size $n$ of the matrix to prove the lemma. For $n=1,2$, the 
	lemma is clearly true. To work on the case $n>2$,
	we assume that the lemma holds for matrices of size less than $n$.
	
	Deleting the $i$-th row and the $j$-th column from $B$, we obtain a matrix $B_{ij}$
	whose determinant multiplied by $(-1)^{i+j}$ is an entry of the adjoint of $B$. Note that $B_{ii}$
	still satisfies the conditions of the lemma. By the induction hypothesis, 
	$\det B_{ii}\geq 0$. If $i<j$, we first perform $j-2$ operations of switching
	rows so that the $(j-1)$-th row of $B_{ij}$ becomes the first row and other rows of $B_{ij}$ 
	keep the order; then we perform $i-1$ operations of switching columns on the new matrix 
	obtained so that $i$-th column becomes the first column and other columns keep the order. 
	After these $i+j-3$ operations, $B_{ij}$ becomes a matrix $B'_{ij}$ satisfying the following 
	conditions.
	\begin{itemize}
		\item
		The first row consists of $\{-\beta_{j\ell}\}_{\ell\neq j}$.
		\item
		The first column consists of $\{-\beta_{\ell i}\}_{\ell\neq i}$.
		\item
	Replacing the entry $-\beta_{ji}$ at the upper left corner of $B'_{ij}$ by $\beta_i$,
		the matrix $B''_{ij}$ obtained satisfies the condition of the lemma.
	\end{itemize}
If $i>j$, we perform $j-1$ operations of switching rows and $i-2$ operations 
	of switching columns. After $i+j-3$ operations of switching rows and columns, $B_{ij}$ also 
	becomes a matrix $B'_{ij}$ satisfying the above three conditions. Now we compute $\det B'_{ij}$ 
	using Laplace expansion on the first row. The first row of $\adj(B'_{ij})$ is the same as that 
	of $\adj(B''_{ij})$. By the induction hypothesis, the entries of the first row of $\adj(B''_{ij})$ are 
	non-negative. Since the first row of $B'_{ij}$ consists of $\{-\beta_{j\ell}\}_{\ell\neq j}$, Laplace 
	expansion shows $\det B'_{ij}\leq 0$. Taking the negative signs from switching rows and 
	columns into account, $\det B'_{ij}=(-1)^{i+j-3}\det B_{ij}$. Hence
	the entry $(-1)^{i+j}\det B_{ij}$ of $\adj(B)$ is non-negative.
	
	To prove $\det B\geq 0$, we may replace $B$ by the matrix obtained by multiplying the $i$-th 
	column by $s_i$ for all $i$. In other words, we may assume that $s_1=\cdots=s_n=1$.
	Now, adding all other columns to the first column of $B$, we obtain a matrix whose entries in 
	the first column are all non-negative. To compute the determinant of the new matrix, we use
	Laplace expansion on the first column. Since all entries of
	$\adj(B)$ are non-negative, $\det B\geq 0$ as well. \qed
\end{proof}

\begin{thm}\label{thm:main}
	A flat numerical semigroup algebra $R'/R$ is a complete intersection, if its Ap\'{e}ry monomials
	form a non-singular rectangle.
\end{thm}

\begin{proof}
	Let $\beta_1\times\cdots\times\beta_n$ be the size of the non-singular rectangle.
	Consider the local $R$-algebra homomorphism $\hat{\pi}\colon R[\![Y_1,\ldots,Y_n]\!]\to R'$,
	where $Y_\ell$ maps to $\bu^{s_\ell}$ for $1\leq\ell\leq n$. The restriction 
	$\pi\colon R[Y_1,\ldots,Y_n]\to R'$ of $\hat{\pi}$ is surjective. Since $\ker\hat{\pi}$ is 
	generated by $\ker\pi$, it suffices to show that $\ker\pi$ is generated by $n$ elements.
	
	We claim that $f_\ell:=Y_\ell^{\beta_\ell}-\bu^{t_\ell}Y_1^{\beta_{\ell 1}}\cdots Y_n^{\beta_{\ell n}}$
	generate $\ker\pi$, where $\bu^{t_\ell}\in R$ and $0\leq\beta_{\ell i}<\beta_i$ for all $i$.
	Let $\bY$ and $\bZ$ be as in (\ref{def:log}). Since $\log_\bY\bZ$ is invertible, 
	we may associate a non-negative number to a monomial $\bY^\bi=Y_1^{i_1}\cdots Y_n^{i_n}$. 
	Let $(j_1,\ldots,j_n)\in\Q^n$ be the vector given by $(j_1,\ldots,j_n)\log_\bY\bZ=(i_1,\ldots,i_n)$. 
	By Lemma~\ref{lem:matrix}, all $j_\ell$ are non-negative. We define
	\[
	\|\log\bY^\bi\|:=j_1+\cdots+j_n.
	\]
	If $i_\ell\geq\beta_\ell$ for some $\ell$, we replace the factor $Y_\ell^{\beta_\ell}$ of 
	$\bY^\bi$ by $\bu^{t_\ell}Y_1^{\beta_{\ell 1}}\cdots Y_n^{\beta_{\ell n}}$ resulting
	$\bu^{t_\ell}\bY^\bi/Z_\ell$. Then
	\[
	\bY^\bi-\bu^{t_\ell}\bY^\bi/Z_\ell=f_\ell \bY^\bi/Y_\ell^{\beta_\ell}\in\langle f_1,\ldots,f_n\rangle.
	\]
	In the logarithmic form, the operation
	subtracts the $\ell$-th row of $\log_\bY\bZ$ from the vector $\bi:=(i_1,\ldots,i_n)$. Hence
	\[
	\|\log(\bu^{t_\ell}\bY^\bi/Z_\ell)\|=\|\log\bY^\bi\|-1.
	\]
	After finitely many such operations, every monomial $Y_1^{i_1}\cdots Y_n^{i_n}$ can be 
	changed so that the condition $i_\ell<\beta_\ell$ is satisfied for each $\ell$. Applying
	these operations, every element of $R[Y_1,\ldots,Y_n]$ can be written as the sum of an 
	element of $\langle f_1,\ldots,f_n\rangle$ and an $R$-linear combination of monomials 
	$Y_1^{i_1}\cdots Y_n^{i_n}$ satisfying $i_\ell<\beta_\ell$. By flatness and the rectangular 
	property, $f_1,\ldots,f_n$ generate $\ker\pi$.  \qed
\end{proof}

If $S$ is a free numerical semigroup, we can find a minimal generator $s$ such that
the Ap\'{e}ry monomials of $\kappa[\![\bu^S]\!]/\kappa[\![\bu^{s}]\!]$ form a non-singular rectangle.
The next example is observed already in \cite[Lemma~3]{W}.
\begin{example}
Let $a$ be an odd positive integer. With  relations
	\[\begin{pmatrix}
	2& -1 & 0 & \cdots & 0 & 0\\
	0& 2 & -1 & \cdots & 0 & 0\\
	0& 0 & 2 & \cdots & 0 & 0\\
	\vdots& \vdots & \vdots & \ddots & \vdots & \vdots\\
	0& 0 & 0 & \cdots & 2 & -1\\
	0& 0 & 0 & \cdots & 0 & 2
	\end{pmatrix}
	\begin{pmatrix}
	2^n+a\\ 2^n+2a\\ 2^n+4a\\ \vdots\\ 2^n+2^{n-2}a\\2^n+2^{n-1}a
	\end{pmatrix}
	=
	\begin{pmatrix}
	2^n\\ 2^n\\ 2^n\\ \vdots\\ 2^n\\ 2^n(a+2)
	\end{pmatrix},
	\]
	the flat rectangular algebra
	$\kappa[\![\bu^{2^n},\bu^{2^n+a},\ldots,\bu^{2^n+2^{n-1}a}]\!]/\kappa[\![\bu^{2^n}]\!]$
	is a complete intersection.
\end{example}

Let $S$ and $T$ be numerical semigroups generated by integers. If $\kappa[\![\bu^S]\!]/\kappa[\![\bu^p]\!]$
and $\kappa[\![\bu^T]\!]/\kappa[\![\bu^q]\!]$ are rectangular for some relatively prime numbers
$p\in S$ and $q\in T$, then $\kappa[\![\bu^{qS+pT}]\!]/\kappa[\![\bu^{pq}]\!]$ is rectangular.
Indeed, 
\[
\begin{array}{l}
\Apr(\kappa[\![\bu^{qS+pT}]\!]/\kappa[\![\bu^{pq}]\!])=\\
\\
\{\bu^{qw_1+pw_2} \,|\, \bu^{w_1}\in
\Apr(\kappa[\![\bu^S]\!]/\kappa[\![\bu^p]\!]) \text{ and } \bu^{w_2}\in\Apr(\kappa[\![\bu^T]\!]/\kappa[\![\bu^q]\!])\}.
\end{array}
\]
See \cite{HK} and also \cite[Proposition~9.11]{GR} for the case of gluing. 
Given rectangles of $\kappa[\![\bu^S]\!]/\kappa[\![\bu^p]\!]$ and 
$\kappa[\![\bu^T]\!]/\kappa[\![\bu^q]\!]$ with matrices $\log_{\bY_1}{\bZ_1}$ and 
$\log_{\bY_2}{\bZ_2}$, the algebra $\kappa[\![\bu^{qS+pT}]\!]/\kappa[\![\bu^{pq}]\!]$
has a rectangle with the matrix 
\[
\log_\bY\bZ=
\begin{pmatrix}\log_{\bY_1}{\bZ_1} & 0\\ 0 & \log_{\bY_2}{\bZ_2}\end{pmatrix}.
\]
Clearly, $\log_\bY\bZ$ is invertible if and only if $\log_{\bY_1}{\bZ_1}$ and $\log_{\bY_2}{\bZ_2}$ 
are invertible. In the absolute case, the existence of rectangles depends on the choice of a Noether 
normalization.  It is possible that $\kappa[\![\bu^{qS+pT}]\!]/\kappa[\![\bu^r]\!]$ is not 
rectangular for any $r\in qS+pT$, even though $\kappa[\![\bu^S]\!]/\kappa[\![\bu^s]\!]$ and 
$\kappa[\![\bu^T]\!]/\kappa[\![\bu^t]\!]$ are rectangular for some $s\in S$ and $t\in T$.

\begin{example}
	Let $S=\langle 2,3\rangle$ and $T=\langle 3,4\rangle$. The algebras 
	$\kappa[\![\bu^S]\!]/\kappa[\![\bu^{12}]\!]$ and
	$\kappa[\![\bu^T]\!]/\kappa[\![\bu^{24}]\!]$ have rectangles
	$\{1,\bu^2,\bu^4\}\times\{1,\bu^3,\bu^6,\bu^9\}$ and
	$\{1,\bu^4,\bu^8\}\times\{1,\bu^3,\bu^6,\bu^9,\bu^{12},\bu^{15},\bu^{18},\bu^{21}\}$, 
	respectively. To see that 
	$\kappa[\![\bu^{7S+5T}]\!]/\kappa[\![\bu^r]\!]$ is not rectangular for any $r\in 7S+5T$, we 
	observe the relation $14+21=15+20$ in $7S+5T=\langle 14,21,15,20\rangle$. If
	$r\in\{14,15,20,21\}$, the product of two minimal monomials is not Ap\'{e}ry. If
	$r\not\in\{14,15,20,21\}$, the product of two minimal monomials equals the product
	of the other two minimal monomials. Both cases can not happen for a rectangle.
\end{example}

To provide more examples, we present a class of flat rectangular algebras.
\begin{prop}
	Let $a$, $b$ and $4$ be integers with greatest common divisor $1$.
	Assume that $\bu^a$ and $\bu^b$ are the minimal monomials of the algebra $\3uRing{4}{a}{b}/\uRing{4}$.
	Then the algebra is rectangular if and only if one of  $a$ or $b$ is even.
\end{prop}
\begin{proof}
	If the algebra is rectangular, then its Ap\'{e}ry monomials are $1,\bu^a,\bu^b$ and $\bu^{a+b}$.
	Assume $a<b$. Since $\bu^a$ and $\bu^b$ are minimal, the monomial $\bu^{2a}$ is not
	Ap\'{e}ry. So $2a=4r+sb$ for some non-negative integers $r$ and $s$. Then $s\leq 1$.
	If $s=0$, then $a=2r$ is even. If $s=1$, then $b$ is even.
	
	Now assume conversely that the
	algebra is not rectangular. Then $\bu^{a+b}$ is not an Ap\'{e}ry monomial.
	Otherwise, the Ap\'{e}ry monomials form the rectangle $\{1,\bu^a\}\times\{1,\bu^b\}$. So $a+b=4r+sa+tb$
	for some non-negative integers $r,s,t$. As $\bu^a$ and $\bu^b$ are minimal, we have $s=t=0$.
	So $a+b$ is even. Since $a$ and $b$ can not be both even, they have to be both odd. \qed
\end{proof}

Next section will provide a partial answer to the following questions.
\begin{question}
Is every flat rectangular algebra a complete intersection? Assume that the
	Ap\'{e}ry monomials of a flat numerical semigroup algebra form a rectangle.
	Is the rectangle always non-singular?
\end{question}

%%%%%%%%%%%%%%%%%%%%%%%%%%%%%%%%%%%%%%%%%%%%%%%%%%%%

\section{Algebras with Few Minimal Monomials}\label{sec:234}

%%%%%%%%%%%%%%%%%%%%%%%%%%%%%%%%%%%%%%%%%%%%%%%%%%%%

In this section, we work on a rectangle of size $\beta_1\times\cdots\times\beta_n$ of a flat numerical semigroup algebra $R'/R$ for the cases $n=2,3,4$. Let $\bu^{s_1},\ldots,\bu^{s_n}$ 
be minimal monomials of $R'/R$. We use the notation
\[
\log_\bY\bZ
\begin{pmatrix} s_1 \\ s_2 \\ \vdots \\s_n \end{pmatrix}=
\begin{pmatrix} t_1 \\ t_2 \\ \vdots \\t_n \end{pmatrix}
\]
as in (\ref{eq:log}), where $Z_\ell=Y_\ell^{\beta_\ell}Y_1^{-\beta_{\ell 1}}\cdots Y_n^{-\beta_{\ell n}}$
and $\bu^{t_\ell}\in R$.

Consider the case $n=2$. From
\[
\begin{pmatrix} \beta_1 & -\beta_{12} \\ -\beta_{21} & \beta_2  \end{pmatrix}
\begin{pmatrix} s_1 \\ s_2  \end{pmatrix}=\begin{pmatrix} t_1 \\ t_2  \end{pmatrix},
\]
we have a relation 
\[
(\beta_1-\beta_{21})s_1+(\beta_2-\beta_{12})s_2=t_1+t_2.
\]
If the numbers $\beta_{21}$ and $\beta_{12}$ were both non-zero, the coefficient $\bu^{t_1+t_2}$
would become the Ap\'{e}ry monomial $\bu^{(\beta_1-\beta_{21})s_1}\bu^{(\beta_2-\beta_{12})s_2}$.
Therefore the matrix $\log_\bY\bZ$ is triangular for $n=2$. The main result of this section is the 
case $n=3$.

\begin{thm}\label{thm:3min}
	For a flat rectangular algebra $R[\![\bu^{s_1},\bu^{s_2},\bu^{s_3}]\!]$, the matrix
	$\log_\bY\bZ$ is triangular  after a suitable permutation of indices. In particular it is non-singular.
\end{thm}
\begin{proof}
	Our proof consists of three steps: (1) $\beta_{ij}\beta_{ji}=0$ for all $i\neq j$.
	(2) $t_i>0$ for some $i$. (3) $\beta_{ij}=\beta_{ik}=0$ for some $\{i,j,k\}=\{1,2,3\}$.
	Then we can change indices to that $\beta_{31}=\beta_{32}=0$. With $\beta_{12}\beta_{21}=0$,
	we may change indices again so that furthermore $\beta_{21}=0$. After these changes of indices,
	the matrix $\log_\bY\bZ$ becomes upper triangular.
	
	{\bf Step 1.}
	We show first that $\beta_{ij}\beta_{ji}=0$ for all $i\neq j$. If not, say 
	$\beta_{12}\beta_{21}\neq 0$, we claim that the conditions
	$\beta_i=\beta_{ji}+\beta_{ki}$ on columns would hold for all $\{i,j,k\}=\{1,2,3\}$.
From
	\[
	\begin{pmatrix} \beta_1 & -\beta_{12} & -\beta_{13} \\ 
	                    -\beta_{21} & \beta_2  & -\beta_{23}\\
	                     -\beta_{31} & -\beta_{32}  & \beta_3\end{pmatrix}
	\begin{pmatrix} s_1 \\ s_2  \\s_3 \end{pmatrix}=\begin{pmatrix} t_1 \\ t_2\\ t_3  \end{pmatrix},
	\]
	we have a relation
	\[
	(\beta_1-\beta_{21})s_1+(\beta_2-\beta_{12})s_2=(\beta_{23}+\beta_{13})s_3+(t_1+t_2).
	\]
	The Ap\'{e}ry monomial $\bu^{(\beta_1-\beta_{21})s_1}\bu^{(\beta_2-\beta_{12})s_2}$ 
	is divisible by the coefficient $\bu^{t_1+t_2}$. Hence $t_1=t_2=0$. If 
	$\beta_{23}+\beta_{13}<\beta_3$, we would have two different factorizations
	\[
	\bu^{(\beta_1-\beta_{21})s_1}\bu^{(\beta_2-\beta_{12})s_2}=\bu^{(\beta_{23}+\beta_{13})s_3}
	\]
	of a monomial in the rectangle. Hence $0\leq\beta_{23}+\beta_{13}-\beta_3<\beta_3$
	and we have another relation
	\[
	(\beta_1-\beta_{21})s_1+(\beta_2-\beta_{12})s_2=
	\beta_{31}s_1+\beta_{32}s_2+(\beta_{23}+\beta_{13}-\beta_3)s_3+t_3.
	\]
	The Ap\'{e}ry monomial $\bu^{(\beta_1-\beta_{21})s_1}\bu^{(\beta_2-\beta_{12})s_2}$ 
	is divisible by the coefficient $\bu^{t_3}$. Hence $t_3=0$ and 
	\[
	\bu^{(\beta_1-\beta_{21})s_1}\bu^{(\beta_2-\beta_{12})s_2}=
	\bu^{\beta_{31}s_1}\bu^{\beta_{32}s_2}\bu^{(\beta_{23}+\beta_{13}-\beta_3)s_3}.
	\]
	Monomials in the rectangle are distinct. Hence the column conditions 
	$\beta_i=\beta_{ji}+\beta_{ki}$ hold for all $\{i,j,k\}=\{1,2,3\}$.

	To get a contradiction from $\beta_{12}\beta_{21}\neq 0$, we work on elements of the form $\beta_is_i+\beta_{jk}s_k$, where $\{i,j,k\}=\{1,2,3\}$. By the column conditions, we have
	\[
	\beta_is_i+\beta_{jk}s_k=(\beta_{ji}+\beta_{ki})s_i+\beta_{jk}s_k=\beta_js_j+\beta_{ki}s_i.
	\]
	Therefore $\beta_is_i+\beta_{jk}s_k$ represents the same number for any $\{i,j,k\}=\{1,2,3\}$. Write 
	\[
	\bu^{\beta_is_i}\bu^{\beta_{jk}s_k}=\bu^t\bu^{\alpha_is_i}\bu^{\alpha_js_j}\bu^{\alpha_ks_k},
	\]
	where $0\leq\alpha_i<\beta_i$, $0\leq\alpha_j<\beta_j$, $0\leq\alpha_k<\beta_k$ and
	$\bu^t\in R$. If $\alpha_i>0$, the Ap\'{e}ry monomial 
	$\bu^{(\beta_i-\alpha_i)s_i}\bu^{\beta_{jk}s_k}$ is divisible by the coefficient $\bu^t$. Then
	$t=0$ and we would have two different factorizations 
	\[
	\bu^{(\beta_i-\alpha_i)s_i}\bu^{\beta_{jk}s_k}=\bu^{\alpha_js_j}\bu^{\alpha_ks_k}
	\]
	of a monomial in the rectangle. Hence $\alpha_i$ vanishes, and so do $\alpha_j$ and 
	$\alpha_k$. Now $\beta_is_i+\beta_{jk}s_k=t$.
	By a similar argument, we have $\beta_is_i+\beta_{kj}s_j=t'\in\log_\bu R$.
	For $i=3$, we obtain a contradiction by two different representations
	\[
	\bu^{t'}\bu^{\beta_{12}s_2}=\bu^t\bu^{\beta_{21}s_1}.
	\]

	{\bf Step 2.} Since $\beta_{12}\beta_{21}=0$, we may assume $\beta_{12}=0$ by changing 
	indices. If $\beta_{13}=0$, then
	$\det\log_\bY\bZ=\beta_1\beta_2\beta_3-\beta_1\beta_{32}\beta_{23}>0$. If $\beta_{13}\neq 0$,
	then $\beta_{31}=0$ and $\det\log_\bY\bZ$ can be computed according to vanishing of $\beta_{32}$:
	If furthermore $\beta_{32}=0$, then 
	$\det\log_\bY\bZ=\beta_1\beta_2\beta_3-\beta_{21}\beta_{12}\beta_3>0$.
	Otherwise, $\beta_{32}\neq 0$ implies that $\beta_{23}=0$ and 
	$\det\log_\bY\bZ=\beta_1\beta_2\beta_3-\beta_{21}\beta_{32}\beta_{13}>0$. 
	In any cases, $\det\log_\bY\bZ>0$. Therefore 
	\[
	\begin{vmatrix}
	\beta_1s_1-\beta_{12}s_2-\beta_{13}s_3 & -\beta_{12} &  -\beta_{13} \\
	-\beta_{21}s_1+\beta_2s_2-\beta_{23}s_3 & \beta_2 &  -\beta_{23} \\
	-\beta_{31}s_1-\beta_{32}s_2+\beta_3s_3 & -\beta_{32} & \beta_3
	\end{vmatrix}
	=s_1\det\log_\bY\bZ>0.
	\]
	Recall that the entries of the adjoint of $\log_\bY\bZ$ are all non-negative. Hence
	$t_i=\beta_is_i-\beta_{ij}s_j-\beta_{ik}s_k>0$ for some $\{i,j,k\}=\{1,2,3\}$.
	
	{\bf Step 3.} 
Provided that $\beta_{ij}s_j+\beta_{ik}s_k=\beta_is_i-t_i>0$ for all 
	$\{i,j,k\}=\{1,2,3\}$, we want to get a contradiction. After changing indices, we may assume that 
	$t_3>0$ and $\beta_{21}=0$. Let $\alpha_1$ be a positive integer in $\log_\bu R$. Then  
	$\alpha:=\alpha_1s_1\in\log_\bu R$. We choose the largest element $s_0\in\log_\bu R$ such
	that there exists a factorization 
	\[
	\bu^\alpha=\bu^{s_0}(\bu^{s_1})^{a_1}(\bu^{s_2})^{a_2}(\bu^{s_3})^{a_3}
	\]
	satisfying the condition $s_0<\alpha$. Write $a_1=n_1\beta_1+a'_1$, where $n_1\in\N$ and 
	$0\leq a'_1<\beta_1$. Let $a'_2:=a_2+n_1\beta_{12}$ and $a'_3:=a_3+n_1\beta_{13}$. 
	Then we have another factorization 
	\[
	\bu^\alpha=\bu^{s_0+n_1t_1}(\bu^{s_1})^{a'_1}(\bu^{s_2})^{a'_2}(\bu^{s_3})^{a'_3}.
	\] 
	Since $\beta_{12}s_2+\beta_{13}s_3>0$, the condition $s_0+n_1t_1<\alpha$ still holds. 
	By the maximality of $s_0$, the number $n_1t_1$ has to vanish. Write $a'_2=n_2\beta_2+a''_2$,
	where $n_2\in\N$ and $0\leq a''_2<\beta_2$. Let $a'''_3:=a''_3+n_2\beta_{23}$. Since we assume
	$\beta_{21}=0$, we have one more factorization 
	\[
	\bu^\alpha=\bu^{s_0+n_2t_2}(\bu^{s_1})^{a'_1}(\bu^{s_2})^{a''_2}(\bu^{s_3})^{a'''_3}.
	\] 
	The condition $s_0+n_2t_2<\alpha$ still holds from the assumption 
	$\beta_{21}s_1+\beta_{23}s_3>0$. By the maximality of $s_0$ again, the number $n_2t_2$ 
	has to vanish. We claim that $a'''_3<\beta_3$. Otherwise, we would have a factorization
	\[
	\bu^\alpha=\bu^{s_0+t_3}(\bu^{s_1})^{a'_1+\beta_{31}}(\bu^{s_2})^{a''_2+\beta_{32}}
	(\bu^{s_3})^{a'''_3-\beta_3}
	\] 
	contradicting the maximality of $s_0$, since we assume $\beta_{31}s_1+\beta_{32}s_2>0$
	and $t_3>0$. As an element in the rectangle,
	the monomial $(\bu^{s_1})^{a'_1}(\bu^{s_2})^{a''_2}(\bu^{s_3})^{a'''_3}$ is Ap\'{e}ry. Now we have  
	different representations  $\bu^\alpha\bu^0$ and $\bu^{s_0}\bu^{a'_1s_1+a''_2s_2+a'''_3s_3}$ 
	of $\bu^\alpha$. This can not happen in a flat algebra.	
	 \qed
\end{proof}

\begin{cor}
	A flat rectangular algebra $R[\![\bu^{s_1},\bu^{s_2},\bu^{s_3}]\!]$ is complete intersection.
\end{cor}

\begin{example}\label{4.4}
	In the flat algebra $\5uRing{16}{24}{31}{46}{44}/\2uRing{16}{24}$, the set of the Ap\'{e}ry monomials
	is a rectangle of size $2\times2\times2$ with the relation
	\[
	\begin{pmatrix}
	2 & -1 & 0  \\
	0& 2 &  -1 \\
	0 & 0 & 2  \\
	\end{pmatrix}
	\begin{pmatrix}31\\46\\44\end{pmatrix}=
	\begin{pmatrix}16\\2\times24\\4\times16+24\end{pmatrix}.
	\]
	The matrix $\log_\bY\bZ$ is triangular.
\end{example}
We define $\log_\bY\bZ$ only for flat rectangular algebras. The following algebra is not flat.
The corresponding $3\times 3$ matrix is singular.
\begin{example}
	The Ap\'{e}ry monomials of the algebra $\3uRing{3}{5}{7}/\2uRing{17}{19}$ are $1$,
$\bu^3$, $\bu^{18}$, $\bu^{21}$ and $\bu^s$ for $5\leq s\leq 16$. They form a
	rectangle of size $4\times2\times2$ with the relation
	\[
	\begin{pmatrix}
	4 & -1 & -1  \\
	-1& 2 &  -1 \\
	-3 & -1 & 2  \\
	\end{pmatrix}
	\begin{pmatrix}3\\5\\7\end{pmatrix}=
	\begin{pmatrix}0\\0\\0\end{pmatrix}.
	\]
\end{example}

Using a result of Bresinsky \cite{B}, we have the following result for
an algebra with $4$ minimal monomials.
\begin{thm}
Let $\bu^{s_1},\bu^{s_2},\bu^{s_3},\bu^{s_4}$ be the minimal monomials of a rectangular
	algebra $R'/R$. If $R=\kappa[\![\bu^s]\!]$ for some $s$ in the semigroup generated by $s_1,s_2,s_3,s_4$,
	then $R'/R$ is complete intersection.
\end{thm}
\begin{proof}
	We may assume that $s_1,s_2,s_3,s_4$ are integers with the greatest common divisor $1$.
In the rectangle of size $\beta_1\times\beta_2\times\beta_3\times\beta_4$, we have 
	relations 
	\[
	\beta_is_i=\beta_{ij}s_j+\beta_{ik}s_k+\beta_{il}s_l+\lambda_is
	\] 
	for $\{i,j,k,l\}=\{1,2,3,4\}$, where $\lambda_i\in\N$. Since flat rectangular algebras are 
	Gorenstein, the semigroup $\langle s_1,s_2,s_3,s_4\rangle$ is symmetric. If the algebra 
	is not complete intersection, then \cite[Theorem 3]{B} says
	\[
(\bu^{s_1})^{\alpha_1}(\bu^{s_3})^{\alpha_3}=(\bu^{s_2})^{\alpha_2}(\bu^{s_4})^{\alpha_4}
	\]
	for some $0<\alpha_i<c_i$, where $c_i=\min\{n \,|\, 0< ns_i\in\langle s_j ; j\neq i\rangle\}$. 
Since monomials in the rectangle are distinct, $\alpha_i\geq\beta_i$ for some $i$.
	Say $\alpha_1\geq\beta_1$. Then $\beta_1<c_1$.
	Write $s=n_1s_1+n_2s_2+n_3s_3+n_4s_4$, where $n_1,n_2,n_3,n_4\in\N$. Then
	\[
	(\beta_1-n_1\lambda_1)s_1=(\beta_{12}+n_2\lambda_1)s_2+
	(\beta_{13}+n_3\lambda_1)s_3+(\beta_{14}+n_4\lambda_1)s_4.
	\]
	By the minimality of $c_1$, the non-negative number $\beta_1-n_1\lambda_1$ 
	has to vanish. Therefore 
	$\beta_{12}=\beta_{13}=\beta_{14}=n_2\lambda_1=n_3\lambda_1=n_4\lambda_1=0$ and 
	$\lambda_1 s=\beta_1s_1=n_1\lambda_1s_1$. Consequently, $\lambda_1>0$ and $s=n_1s_1$.
	The monomial 
	$(\bu^{s_1})^{\alpha_1}=\bu^{\lambda_1 s}(\bu^{s_1})^{\alpha_1-\beta_1}$ is not Ap\'{e}ry,
	nor is $(\bu^{s_2})^{\alpha_2}(\bu^{s_4})^{\alpha_4}=(\bu^{s_1})^{\alpha_1}(\bu^{s_3})^{\alpha_3}$.
	Therefore $\alpha_i\geq\beta_i$ for $i=2$ or $4$. Say $\alpha_2\geq\beta_2$. Then 
	$\beta_2<c_2$ and $\lambda_2s=\beta_2s_2$ by the same argument as above. Now the relation
	$\beta_2s_2=\lambda_2n_1s_1$ contradicts the minimality of $c_2$.
	\qed
\end{proof}

\paragraph{\bf{Acknowledgements}}
Part of this work has been developed during the visits of I-Chiau Huang to the Institute for
Research in Fundamental Sciences (IPM) and Raheleh Jafari to the Institute of Mathematics,
Academia Sinica (IoM) in 2016. Both authors would like to thank IPM and IoM for their 
hospitality and support. The authors are grateful to Mee-Kyoung Kim for inspiring 
conversation and comments. Finally we would like to thank the referee for a careful reading 
of the manuscript and  valuable comments and suggestions.

\end{document}